\documentclass{amsart}
\usepackage{}

\newtheorem{theorem}{Theorem}[section]

\theoremstyle{definition}
\newtheorem{definition}[theorem]{Definition}
\newtheorem{example}[theorem]{Example}

\newtheorem{proposition}[theorem]{Proposition}

\theoremstyle{remark}
\newtheorem{remark}[theorem]{Remark}

\numberwithin{equation}{section}

\newcommand{\Z}{\mathbb{Z}}
\newcommand{\R}{\mathbb{R}}
\newcommand{\C}{\mathbb{C}}

\newcommand{\eps}{\varepsilon}

\newcommand{\g}[1]{{\mathfrak{#1}}}
\newcommand{\im}{\mathrm{Im}}
\newcommand{\re}{\mathrm{Re}}

\newcommand{\ip}[2]{\langle #1,#2\rangle}
\newcommand{\dotcup}{\ensuremath{\mathaccent\cdot\cup}}

\makeatletter
\def\moverlay{\mathpalette\mov@rlay}
\def\mov@rlay#1#2{\leavevmode\vtop{%
   \baselineskip\z@skip \lineskiplimit-\maxdimen
   \ialign{\hfil$\m@th#1##$\hfil\cr#2\crcr}}}
\newcommand{\charfusion}[3][\mathord]{
    #1{\ifx#1\mathop\vphantom{#2}\fi
        \mathpalette\mov@rlay{#2\cr#3}
      }
    \ifx#1\mathop\expandafter\displaylimits\fi}
\makeatother

\newcommand{\bigcupdot}{\charfusion[\mathop]{\bigcup}{\cdot}}

\begin{document}

\title[Hypergeometric Shift operators]{An application of 
hypergeometric shift operators to the $\chi$-spherical Fourier transform}


\author{Vivian M. Ho}
\address{Department of Mathematics, Louisiana State University, Baton
Rouge, Louisiana 70803}
\email{vivian@math.lsu.edu}

\author{Gestur \'{O}lafsson}
\address{Department of Mathematics, Louisiana State University, Baton
Rouge, Louisiana 70803}
\email{olafsson@math.lsu.edu}
\thanks{The research of G. \'Olafsson was supported by NSF grant DMS-1101337}  

\date{}


\begin{abstract}

We study the action of hypergeometric shift operators on the Heckman-Opdam hypergeometric functions associated with the
$BC_n$ type root system and some negative multiplicities. Those hypergeometric functions are connected to the
$\chi$-spherical functions on Hermitian symmetric spaces $U/K$ where $\chi$ is a nontrivial character of $K$.
We apply shift operators to the hypergeometric functions to move negative multiplicities to positive ones.
This allows us to use many well-known results of the hypergeometric functions associated with positive multiplicities. 
In particular, we use this technique to achieve exponential estimates for the $\chi$-spherical functions. 
The motive comes from the Paley-Wiener type theorem on line bundles over Hermitian symmetric spaces. 

\end{abstract}

\maketitle

\section{Introduction}
\label{introduction}

\noindent
The theory of spherical functions on semisimple Riemannian symmetric spaces is an interesting part of harmonic analysis 
dating back to the work of Gel'fand, Godement (for the abstract setting), 
Harish-Chandra (in the concrete setting for a Riemannian symmetric space), and Helgason. 
Later this theory was generalized as the theory of hypergeometric functions, resp. of Jacobi polynomials, 
of several variables associated with a root system, by a series of joint work of Heckman and Opdam \cite{hop, he, op1, op2}. 
In the case of spherical functions one can investigate them using both differential and 
integral operators, while there are only the differential equations at hand for the case of hypergeometric functions. 
The Heckman-Opdam hypergeometric functions are joint eigenfunctions of a commuting algebra of differential operators  associated to a root system and a multiplicity parameter (which is a Weyl group invariant function 
on the root system). The multiplicities can be arbitrary complex numbers.
These hypergeometric functions are holomorphic, Weyl group invariant, and normalized by the value one at the identity.
When the root multiplicities do correspond to those of a Riemannian symmetric space, these hypergeometric functions are
nothing but the restrictions to a Cartan subspace of spherical functions on the associated Riemannian symmetric space. 
For an overall study of this subject we refer to the books \cite[Part I]{hs} and \cite{h2}. 

One of the new tools which was born with the general theory of hypergeometric functions was
the theory of shift operators. Those are generalizations of
the classical identity for hypergeometric functions 
\begin{equation}
\label{eq:so1}
\frac{d}{d z}  F  (\alpha, \beta, \gamma;  z) = \frac{\alpha  \beta}{\gamma}  F  (\alpha + 1, \beta + 1, \gamma+1;  z).
\end{equation}
Here, $\frac{d}{d z}$ is the simplest example of a shift operator, for the rank one root system $BC_1$. 
A shift operator for the root system of type $BC_2$ was first found by Koornwinder in \cite{ko}, 
used to study the Jacobi polynomials. 
Subsequently some particular higher rank cases were established by several authors. A thorough study of higher rank shift 
operators in full generality was done by Opdam in \cite{op1, op2, op3}. We also recommend \cite[Part I]{hs} as a good resource.

In this article we mainly explore the idea to apply some suitable shift operators to hypergeometric functions 
associated with the root system of type $BC_n$ ($n \geq 1$) and certain negative multiplicities. 
The motivation for this work originates from our article \cite{hol}
on Paley-Wiener type theorems for line bundles over compact
Hermitian symmetric spaces $U/K$ and the needed estimates
for the spherical functions.  Let $\chi$ be a nontrivial character of $K$. 
We characterized the $\chi$-bicovariant functions $f$ on $U$ (geometrically equivalent to smooth sections of those line bundles) 
with small support in terms of holomorphic extendability and exponential growth of their $\chi$-spherical Fourier 
transforms with the exponent linked to the size of the support of $f$.  This
characterization relies on the fact that the $\chi$-spherical functions on $U$ extend holomorphically to their complexifications,
and their restrictions on the noncompact dual $G$ are in turn the $\chi$-spherical functions on $G$. 
It is well known, see \cite{hs}, that the $\chi$-spherical functions on $G$ are 
related to hypergeometric functions with shifted multiplicities, some of which can be negative.

Denote for a moment a hypergeometric function by $F$.
In \cite{op} the author gave a uniform exponential estimate on the growth
behavior of $F$, which is of crucial importance for the Paley-Wiener Theorem, but which requires all multiplicities  to
be positive, see Proposition \ref{pro1}. There are two possible ways now to attach the problem. Either
generalize the Opdam estimates in \cite{op} or use the shift operators to reduce the problem to 
positive multiplicities. The first way was chosen in \cite{hol}. Here we discuss the second idea.

Let us give a brief outline of this article. In Section \ref{pre} we settle on some necessary definitions and notations. 
A succinct review of the Heckman-Opdam hypergeometric functions is given in Section \ref{hypergeo}. 
To follow, in Section \ref{shift} we study some properties of the hypergeometric shift operators of Opdam that will be used in 
the subsequent sections. Next, Section \ref{appl} is devoted to an application of shift operators to the 
hypergeometric functions associated with the root system of type $BC_n$ and certain nonpositive multiplicities, where root multiplicities 
do correspond to Hermitian symmetric spaces. Using this technique we will achieve desired estimates for the $\chi$-spherical functions.
Finally, in Section \ref{roc} we treat the rank one case as an example to strengthen the skills. A remark is given for a nice generalization of 
the rank one case. 

\section{Notation and Preliminaries}
\label{pre}

\noindent
The material in this section is standard. 
We refer to \cite[Part I]{hs} for basic notations and definitions.
We use the notations from the introduction mostly without reference.

Let $\g{a}$ be an Euclidean space of dimension $n$ and $\g{a}^\ast$ its dual space. Denote by $\langle\, \cdot\, ,\, \cdot\, \rangle$
the inner products on $\g{a}$ and $\g{a}^\ast$. Let $\Sigma \subset \g{a}^\ast$ be a possibly nonreduced root system with 
$\mathrm{rank}  (\Sigma) = \dim  \g{a} = n$. In particular $\Sigma$ spans $\g{a}^\ast$. 
Denote by
\begin{equation}\label{eq13}
\Sigma_\ast = \{\alpha \in \Sigma \mid 2 \alpha \notin \Sigma\} \quad\text{ and }\quad \Sigma_{\text{i}} = \{\alpha \in \Sigma \mid
\frac{1}{2}  \alpha \notin \Sigma\}  .
\end{equation}
Both $\Sigma_\ast$ and $\Sigma_i$ are reduced root systems. Fix a system $\Sigma^+$ of positive roots in $\Sigma$. 
Set $\Sigma_\ast^+ = \Sigma_\ast \cap \Sigma^+$. Denote by $\g{a}_\C$ the complexification of $\g{a}$:
\[\g{a}_\C = \g{a} \oplus \g{b} = \g{a} \otimes_\R \C,\quad \g{b} = i  \g{a}\]
and by $A_\C$ the complex torus with Lie algebra $\g{a}_\C$. We have the polar decomposition $A_\C = A  B$ with $A = \exp  \g{a}$ 
the split form and $B = \exp  \g{b}$ the compact form. 

The Weyl group $W = W  (\Sigma)$ is generated by the reflections $r_\alpha$ for $\alpha \in \Sigma$. 
A multiplicity function $m:\Sigma \to \C$ is a $W$-invariant function. Write $m_\alpha = m (\alpha)$ 
\footnote{Our multiplicity notation is different from the one used by Heckman and Opdam.
The root system $R$ they use is related to our
$2\Sigma$, and the multiplicity function $k$ in Heckman and Opdam's work is related to our $m$ by 
$k_{2\alpha}=\frac{1}{2}m_\alpha$.}.
Set $m_\alpha = 0$ if $\alpha \notin \Sigma$. A multiplicity function
is said to be positive if $m_\alpha \ge 0$ for all $\alpha$. 
The set of multiplicity functions is denoted by $\mathcal{M}$ and the subset
of positive multiplicity functions is denoted by $\mathcal{M}^+$. Let 
\[\rho = \rho  (m) = \frac{1}{2}  \sum_{\alpha \in \Sigma^+}  m_\alpha  \alpha.\]

For $\lambda \in \g{a}_\C^\ast$ and $\alpha \in \g{a}^\ast$ with $\alpha \ne 0$, set 
\[\lambda_\alpha := \frac{\langle \lambda,  \alpha \rangle}{\langle \alpha,  \alpha \rangle}.\]
Let $P$ be the weight lattice of $\Sigma$, that is, 
\[P = \{\lambda \in \g{a}^\ast\; \mid\; \lambda_\alpha \in \Z,\; \forall  \alpha \in \Sigma\}.\]
Write $\exp: \g{a}_\C \to A_\C$ for the exponential map and $\log: A_\C \to \g{a}_\C$ the multi-valued inverse. 
For $\lambda \in \mathfrak{a}_{\mathbb{C}}^\ast$, we define the function $e^\lambda$ on $A_\mathbb{C}$ by 
\begin{equation}
\label{eq:exp}
e^\lambda  (a) = a^\lambda := e^{\lambda  (\log  a)}.
\end{equation}
If $\lambda \in P$, we see that $e^\lambda$ is well defined and single-valued on $A_\mathbb{C}$. 
So $a \mapsto a^\lambda$ ($\lambda \in P$) is a character of $A_\mathbb{C}$. 
Since $\exp: \g{a} \to A$ is a bijection, (\ref{eq:exp}) is well defined and single-valued on $A$ for all $\lambda \in \g{a}_\C^\ast$.
Denote by $\mathbb{C}  [P]$ (or $\mathbb{C} 
[A_\mathbb{C}]$) the $\mathbb{C}$-linear span of
$e^\lambda$ with $\lambda \in P$ satisfying $e^\lambda \cdot e^\mu = e^{\lambda + \mu}$, $(e^\lambda)^{-1} = e^{- \lambda}$, and
$e^0 = 1$. An element of $\mathbb{C}  [P]$ is an exponential polynomial on $A_\mathbb{C}$ of the form $\sum_{\lambda \in P} 
c_\lambda  e^\lambda$ where $c_\lambda \in \mathbb{C}$ and $c_\lambda \ne 0$ for only finitely many $\lambda \in P$. 

\section{The Heckman-Opdam Hypergeometric Functions}
\label{hypergeo}

\noindent
In this section we mainly review some facts of the theory of Heckman-Opdam hypergeometric functions which will be used later.
The Harish-Chandra expansion was the basic tool in Heckman and Opdam's theory of (generalized) hypergeometric functions for
arbitrary multiplicity functions. For a complete construction of the Heckman-Opdam hypergeometric functions, see 
\cite{hop, he, op1, op2} or \cite[Part I, Chapter 4]{hs}.

Let $\{\xi_j\}_{j = 1}^n$ be an orthonormal basis for $\mathfrak{a}$ and
$\partial_{\xi_j}$ the corresponding directional derivative with respect to $\xi_j$, i.e. 
\[(\partial_{\xi_j}  \phi)  (X) = \frac{d}{d t}  \phi  (X + t  \xi_j)  \Big|_{t = 0}.\]
We define a modified operator $ML  (m) = L  (m) + \langle \rho(m),  \rho(m) \rangle$, where
\[L (m) := \sum_{j = 1}^n  \partial_{\xi_j}^2 + \sum_{\alpha \in \Sigma^+}  m_\alpha  \frac{1 + e^{- 2  \alpha}}{1 - e^{-
2  \alpha}}  \partial_\alpha.\]

Let $\Xi := \{\sum_{j=1}^n n_j \alpha_j\, \mid\, n_j \in \mathbb{Z}^+, \alpha_j \in \Pi\}$ (here $\Pi$ is the set of
simple roots in $\Sigma^+$). 
The Harish-Chandra series corresponding to a multiplicity function $m$ is defined by 
\[\Phi (\lambda, m; a) = a^{\lambda - \rho} \sum_{\mu \in \Xi} \Gamma_\mu (\lambda, m)\, a^{- \mu},\qquad a \in A^+\]
where $A^+ = \{a \in A\, \mid\, e^\alpha (a) > 1, \forall \alpha \in \Sigma^+\}$ and $\lambda \in \g{a}_\C^\ast$. 
The coefficients $\Gamma_\mu (\lambda, m) \in \C$ are uniquely determined by $\Gamma_0 (\lambda, m) = 1$ 
and the recurrence relations (using the eigenvalue equation of $L (m)$)
\[\ip{\mu}{\mu - 2 \lambda} \Gamma_\mu (\lambda, m) = 2 \sum_{\alpha \in \Sigma^+} m_\alpha \sum_{\substack{k \in \mathbb{N} 
\\ \mu - 2 k \alpha \in \Xi}}
\Gamma_{\mu - 2 k\alpha} (\lambda, m) \ip{ \mu + \rho - 2 k \alpha - \lambda}{\alpha}\]
provided $\lambda$ satisfies $\ip{ \mu}{ \mu - 2 \lambda} \ne 0$ for all $\mu \in \Xi \setminus \{0\}$.

\begin{definition}
Define the meromorphic functions $\widetilde{c}, c: \mathfrak{a}_{\mathbb{C}}^\ast \times \mathcal{M} \to \mathbb{C}$ by
\[c  (\lambda,  m) = \frac{\widetilde{c}  (\lambda,  m)}{\widetilde{c}  (\rho,  m)},\quad  \widetilde{c}  (\lambda,  m)
= \prod_{\alpha \in \Sigma^+}  \frac{\Gamma  (\lambda_\alpha + \frac{m_{\alpha/2}}{4})}{\Gamma  (\lambda_\alpha +
\frac{m_{\alpha/2}}{4} + \frac{m_\alpha}{2})}\]
where $\Gamma$ is the Euler Gamma function given by $\Gamma  (x) = \int_0^\infty t^{x-1} e^{-t}   \, d t$. 
The function $c$ is the well-known Harish-Chandra's $c$-function.
\end{definition}

We note that
\[c(\lambda ,m)c(-\lambda, m)= \prod_{\alpha \in \Sigma}  \frac{\Gamma  (\lambda_\alpha + \frac{m_{\alpha/2}}{4})}{\Gamma  (\lambda_\alpha +
\frac{m_{\alpha/2}}{4} + \frac{m_\alpha}{2})}\]
is clearly $W$-invariant in $\lambda$ whereas $c(\lambda ,m)$ is not.

\begin{definition}
The function
\begin{equation}
\label{eq:cfun}
F (\lambda, m; a) = \sum_{w \in W} c (w \lambda, m) \Phi (w \lambda, m; a)
\end{equation}
is called the \textit{hypergeometric function} on $A$ associated with the triple $(\g{a},  \Sigma,   m)$. 
\end{definition}

\begin{remark}
In the theory of spherical functions the equation (\ref{eq:cfun}) is the well
known Harish-Chandra's asymptotic expansion 
for the spherical function \cite{hc}. In that case 
an explicit expression for Harish-Chandra's $c$-function was given  by 
Gindikin and Karpelevi\v{c} \cite{gk}.
\end{remark}

\begin{theorem}
\label{thm1}
Let $\mathcal{P} \subset \mathcal{M}$ be defined by
\[\mathcal{P} = \{m \in \mathcal{M}\; \mid\; \widetilde{c}  (\rho,  m) = 0\}.\]
Then $F  (\lambda,  m;  a)$ is holomorphic in $\mathfrak{a}_\mathbb{C}^\ast \times (\mathcal{M} \setminus \mathcal{P}) \times T$, 
where $T$ is a $W$-invariant tubular neighborhood of $A$ in $A_\mathbb{C}$, and $F  (\lambda,  m;  a)$ is $W$-invariant 
on the same domain (in the variables $\lambda$ and $a$).
\end{theorem}

\begin{proof}
See Theorem $4.4.2$ in \cite{hs}.
\end{proof}

\begin{proposition}
\label{pro1}
Let $m \in \mathcal{M}^+$. There exists a constant $C$ such that 
\[|F  (\lambda,  m;  \exp  (X + i  Y)| \leq C  \exp  (\max_{w \in W}  \re  w  \lambda  (X) - \min_{w \in W}  
\im  w  \lambda  (Y))\]
where $X, Y \in \g{b}$ with $|\alpha  (X)| \leq \pi/2$ for all $\alpha \in \Sigma$ and $\lambda \in \g{a}_\C^\ast$.
\end{proposition}

\begin{proof}
This follows from Proposition 6.1 in \cite{op}.
\end{proof}

\section{Hypergeometric Shift Operators}
\label{shift}

\noindent
We will introduce the hypergeometric shift operators of Opdam and discuss an example of such a shift operator when the 
root system is of type $BC_n$. 
For a comprehensive information about these shift operators we recommend the reader to check 
\cite{op1, op2, op3} or \cite[Part I, Chapter 3]{hs}.

We keep the notations as in Section \ref{pre}. For the moment denote by $k$ a multiplicity in $\mathcal{M}$.
The hypergeometric shift operators with shift $k$ were
constructed in \cite{op1, op2}. They are differential operators  $D  (k)$ on $A_\C$ (or $\g{a}_\C$) satisfying the
commuting relation  
\[D  (k)  \circ ML  (m) = ML  (m + k) \circ D  (k), \quad \forall m \in \mathcal{M}.\]
In particular shift operators are $W$-invariant on $A_\C^{\mathrm{reg}}$ 
where $A_\C^{\mathrm{reg}} = \{a \in A_\C\, \mid\, w a \ne a,\, \forall w \in W,\, w \ne e\}$
(cf. \cite[Corollary 3.1.4]{hs}).

\begin{definition} \label{def1} (cf. \cite[Def 3.4.1]{hs})
Let $\Sigma = \dotcup_{i = 1}^\ell  \mathcal{O}_i$ be the disjoint union of $W$-orbits in $\Sigma$ where 
$\ell \in \mathbb{Z}^+$ is the number of  Weyl group orbits in $\Sigma$. Define
$e_i \in \mathcal{M}$ by $e_i  (\alpha) = \delta_{i  j}$ (Kronecker's symbol) for all $\alpha \in \mathcal{O}_j$. Let
$\mathcal{B} = \{b_1,  \cdots,  b_\ell\}$ be the basis of $\mathcal{M}$ given by
\[b_i = 
\begin{cases}
2  e_i & \text{if}\; 2  \mathcal{O}_i  \cap  \Sigma = \emptyset\\
4  e_i - 2  e_j & \text{if}\; 2  \mathcal{O}_i = \mathcal{O}_j\; \text{for some}\;  j.
\end{cases}\]
Note that $m \in \mathcal{M}$ is integral if and only if $m \in \mathbb{Z}  \mathcal{B}$. A shift operator with shift 
$k \in \mathbb{Z} \mathcal{B}$ is called a raising operator if $k \in \mathbb{Z}^+  \mathcal{B}$ and a lowering operator 
if $k \in \mathbb{Z}^- \mathcal{B}$.
\end{definition}

\begin{remark}
The open set $\mathcal{M} \setminus \mathcal{P}$ contains the closed subset 
$$\{m \in \mathcal{M}\; \mid\; \re\; (m_{\alpha/2} + m_\alpha) \geq 0,\; \forall  \alpha \in \Sigma_\ast\}$$
which in turn contains the  set $\mathbb{C}^+  \mathcal{B}$. 
\end{remark}

Denote by $D^\ast$ the formal transpose of a differential operator $D$ on $A_\mathbb{C}$ with respect to the Haar measure $d  a$ on $A$:
$$\int_A  (D  f(a))\; g(a)  \, d a = \int_A f (a) \; (D^\ast  g (a)) \, d a.$$
In a same way we define the transpose of a differential operator on $B$.
Define the weight function $\delta$ by 
\[\delta = \delta  (m) = \prod_{\alpha \in \Sigma^+}  (e^\alpha - e^{- \alpha})^{m_\alpha}.\]

\begin{theorem}
\label{thm3}
(Existence of shift operators).
\begin{enumerate}
\item There exist nontrivial shift operators of shift $k$ if and only if $k \in \mathbb{Z}  \mathcal{B}$. 
\item Let $k \in \mathbb{Z}^+  \mathcal{B}$ and $m \in \mathcal{M}$ with $m, m \pm k \notin \mathcal{P}$. Then there is a unique
shift operator $G_-$ of shift $- k$ (a lowering operator) such that
\begin{equation}
\label{eq10}
G_-  (- k,  m)  F  (\lambda,  m;\; \cdot  ) = F  (\lambda,  m - k;\; \cdot  ).
\end{equation}
Define
\[G_+  (k,  m) := \delta  (- k - m) \circ G_-^\ast  (- k,  m + k) \circ \delta  (m).\]
It is a raising operator with shift $k$.
\end{enumerate}
\end{theorem}

\begin{proof}
The existence of nontrivial shift operators is asserted by Theorem 3.6 in \cite{op2}. 
Also see Theorem 3.4.3 and Corollary 3.4.4 in \cite{hs}. 
So now there is a lowering operator $G_-  (- k,  m)$ which satisfies the formula (4.4.8) in \cite{hs}. Multiplying it by 
\[\frac{\widetilde{c}  (\rho  (m),  m)}{\widetilde{c}  (\rho (m - k),  m-k)},\]
which is a constant depending on $m$ and $k$, and has no poles by the assumption that $m, m \pm k \notin \mathcal{P}$, 
the derived operator is again a lowering operator with shift $-k$ and satisfies (\ref{eq10}).
\end{proof}

\begin{example}
\label{exm2}
Let $\{\eps_j\}_{j=1}^n$ be an orthonormal basis of $\g{a}^\ast$. We consider an irreducible root system of type $BC_n$:
\[\Sigma=\pm \{\eps_i, 2\eps_i ,\eps_j \pm \eps_k \mid 1 \leq i \leq n, 1 \leq j < k \leq n\}.\]
Then the $\eps_i$ are short, $2\eps_i$ are long and the roots $\eps_j\pm \eps_k$ are medium. We have
\[\Sigma = \bigcupdot_{i = 1}^3 \mathcal{O}_i = \mathcal{O}_s \bigcupdot \mathcal{O}_m \bigcupdot \mathcal{O}_l\]
where the three disjoint $W$-orbits correspond to short, medium, and long roots, respectively.
Similarly, $(m_s, m_m, m_l) \in \mathcal{M}$ are multiplicities with respect to each of them. 
Note that some of the multiplicities are allowed to be zero. 
If $\alpha \in \mathcal{O}_s$, then  $2  \alpha \in \mathcal{O}_l$. So by Definition \ref{def1},
\[b_1 = 4  e_1 + 0  e_2 - 2  e_3\]
which means shifting up the multiplicity of short roots by $4$, shifting down on long roots by $2$, and no change on medium
roots, that is, $b_1 = (4,  0,  - 2)$. Similarly,  
\[b_2 = 0 + 2  e_2 + 0 = (0,  2,  0),\quad b_3 = 0 + 0 + 2  e_3 = (0,  0,  2).\]
Thus, $\mathcal{B} = \{b_1, b_2, b_3\}$ forms a basis of $\mathcal{M}$ (associated to $\Sigma$). 

As an example of shifts in the multiplicity parameters of the hypergeometric functions using shift operators
we discuss here our fundamental example. 
Let $l \in \mathbb{Z}$. Let $k = |l|  b_1 \in \mathbb{Z}^+  \mathcal{B}$. We start with a multiplicity 
parameter
\[m' = (m_s + 2  |l|,  m_m,  m_l) \in \mathcal{M} \setminus \mathcal{P}.\]
It satisfies 
\begin{eqnarray*}
m_+(l) & := & (m_s - 2  |l|,  m_m,  m_l + 2  |l|) \\
& = & (m_s + 2  |l|,  m_m,  m_l) - |l|  (4, 0, - 2) \\
& = & m' - k.
\end{eqnarray*}
If $m'-k \notin \mathcal{P}$, then since $k \in \Z^+  \mathcal{B}$, 
by (\ref{eq10}), there exists a shift operator $G_-$ such that for all 
$\lambda \in \mathfrak{a}_{\mathbb{C}}^\ast$,
\begin{equation}
\label{eq11}
F  (\lambda,  m_+(l);\; \cdot  ) = G_-  (- k,  m')  F  (\lambda,  m';\; \cdot  ), \quad \text{with}\; k = |l|  b_1.
\end{equation}
In short, applying $G_-(-k,m')$ to the hypergeometric function with parameter
$m'$ gives the hypergeometric function with parameter $m_+(l)$.
\end{example}

\begin{example} \label{exm1} (cf. Example 1.3.2 and Proposition 3.3.1 in \cite{hs})
We work on the rank one root system of type $BC_1$, say $\Sigma = \{\pm  \alpha,  \pm  2  \alpha\}$. 
Fix $\Sigma^+ = \{\alpha,  2 \alpha\}$ with $\langle\alpha,  \alpha\rangle = 1$. 
Set $k_1 = m_{\alpha}$ and $k_2 = m_{2  \alpha}$. The Weyl group $W$ associated to this $\Sigma$ is just $\{\pm  1\}$. 
It is well-known that $\mathbb{C}  [P] = \mathbb{C}  [x,  x^{-1}]$ with $x = e^{2  \alpha}$, and
$\mathbb{C}  [x,  x^{-1}]^W = \mathbb{C}  [s]$ with $s = (x + x^{-1})/2$. We have the following shift operator with shift 
$-b_1 = (-4,  0,  2)$ (in terms of the coordinate $x$)
\begin{equation}
\label{eq19}
E_- = E_- (- b_1,  m) = \frac{1 - x^{-1}}{1 + x^{-1}}  x  \frac{d}{d  x} + C,\quad C = k_1 + k_2 - 1.
\end{equation}
\end{example}

For any $m \in \mathcal{M}$ and all $F,  H \in \mathbb{C}  [P]$, define an inner product
\begin{eqnarray*}
(F,  H)_m & = & \int_B  F  (x)  \overline{H  (x)}  |\delta  (m,  x)|\,  d x\\
& = & |W| \int_{B^+}  F  (x)  \overline{H  (x)}  \delta  (m,  x)\, d x
\end{eqnarray*}
provided that the integral exists. Here $|W|$ is the number of elements in $W$ and $B^+$ is a positive Weyl chamber 
associated to a choice of positive roots. Thus $B^+$ is open, $W\cdot B^+$ is open, dense and of
full measure in $B$ and $dx$
is the Haar measure on $B$ normalized by $\int_B\, d x = 1$.

\begin{proposition}
\label{pro6}
If $F, H \in \mathbb{C} [P]^W$, then
\[(G_- (- k,  m)  F, H)_{m - k} = (F,  G_+  (k,  m - k)  H)_m,\; \forall  k \in \mathbb{Z}^+  \mathcal{B}.\]
\end{proposition}

\begin{proof}
We  have
\begin{eqnarray*}
& & (G_- (- k, m) F, H)_{m - k} \\
&  = & |W| \int_{B^+}\, [G_- (- k, m) F (x)] \overline{H (x)} \delta (m - k, x)\, d x\\
& = & (F, G_-^\ast (- k, m) H)_{m - k}\\
&  = & |W| \int_{B^+}\, F (x) \overline{G_-^\ast (- k, m)\, H (x)} \delta (m - k, x)\, d x\\
& = & |W| \int_{B^+}\, F (x) \overline{G_-^\ast (- k, m) \circ \delta (m - k) H (x)}\, d x\\
& = & |W| \int_{B^+}\, F (x) \overline{\delta (- m) \circ G_-^\ast (- k, m) \circ \delta (m - k)  H (x)}
\delta (m, x)\, d x\\
& = & |W| \int_{B^+}\, F (x) \overline{G_+ (k, m - k) H (x)} \delta (m, x)\, d x\\
& = & (F, G_+ (k, m - k)\, H)_m,
\end{eqnarray*}
where on $B^+$, $G_+ (k, m-k) := \delta  (- m) \circ G_-^\ast  (- k,  m) \circ \delta  (m - k)$ is a differential 
operator (note: $G_+$ is a raising operator with shift $k$, the same as in Theorem \ref{thm3}). 
\end{proof}

\section{An Application to the $\chi$-spherical functions}
\label{appl}

\noindent
We consider the situation when root multiplicities correspond to those of an irreducible Hermitian 
compact symmetric space. Hence $\Sigma$ is of type $BC_n$ and $m_l = 1$.
Then we describe how to apply shift operators to the $\chi$-spherical functions which are connected to the 
hypergeometric functions associated with root system $\Sigma$ and certain nonpositive multiplicities (cf. Proposition
\ref{pro2}).

Let us briefly review some necessary notations and facts about symmetric spaces. The book \cite{h1} is served as a good reference. 
We keep most of notations as in \cite{hol} and refer to it for more information.

Let $G$ be a noncompact semisimple connected Lie group with Lie algebra $\g{g}$. Denote by $\theta$ a Cartan involution on $G$.
Write $\g{g} = \g{k} \oplus \g{s}$ where $\g{k} = \g{g}^\theta$ and $\g{s} = \g{g}^{- \theta}$ for the decomposition of 
$\g{g}$ into $\theta$-eigenspaces. Let $K = G^\theta$. The compact dual of $G$ is denoted by $U$. The space $G/K$ is a Riemannian 
symmetric space of the noncompact type, and its dual $U/K$ is the one of compact type. Let $\g{a}$ be a maximal abelian subspace 
in $\g{s}$. Then $n = \dim  \g{a}$ is the rank of $G/K$. 
Let $\Sigma \subset \g{a}^\ast \setminus \{0\}$ be a system of roots of $\g{a}$ in $\g{g}$. Choose a positive root system $\Sigma^+$
in $\Sigma$. Set $\g{b} = i  \g{a}$. Let $A = \exp  \g{a}$ and $B = \exp  \g{b}$.

We assume, in addition, that $G/K$ (resp. $U/K$) is an irreducible Hermitian symmetric space. 
Write $m = (m_s, m_m, m_l)$ for the multiplicities of roots in $\Sigma$ as before. In our case, $m_l = 1$.

Choose a generator $\chi_1$ for the set of one-dimensional characters of $K$. Then a (nontrivial) character of $K$ 
is given by $\chi = \chi_l = \chi_1^l$ for some $l \in \Z$. 
The $\chi$-spherical functions $\varphi_{\lambda,  l}$ on $G$ with spectral parameters 
$\lambda \in \g{a}_\C^\ast$ can be defined 
by integral or differential equations (cf. \cite[Section 4]{hol}). These functions are closely related to the Heckman-Opdam 
hypergeometric functions in the sense which will be explained as follows. 
Let $\mathcal{O}_s^+$ be the set of short roots in $\Sigma^+$. Define 
\[\eta_l^\pm = \prod_{\alpha \in \mathcal{O}_s^+}  \left(\frac{e^\alpha + e^{- \alpha}}{2}\right)^{\pm  2 |l|}.\]
Note that $\eta_l^+$ is holomorphic on $A_\mathbb{C}$.

\begin{proposition}
\label{pro2}
In case $G/K$ is an irreducible Hermitian symmetric space we have
\[\varphi_{\lambda,  l}  |_A = \eta_l^\pm  F  (\lambda,  m_\pm (l);\; \cdot  )\]
where $\lambda \in \mathfrak{a}_{\mathbb{C}}^\ast$, $m_\pm (l) \in \mathcal{M} \cong \mathbb{C}^3$ is given by
\[m_\pm (l) = (m_s \mp 2  |l|,\; m_m,\; m_l \pm 2  |l|),\]
and the $\pm$ sign indicates that both possibilities are valid. 
\end{proposition}

\begin{proof}
See \cite[p.76, Theorem 5.2.2]{hs} for the proof of this proposition. 
\end{proof}

Note that the space of smooth sections of homogeneous line bundles $\mathcal{L}_\chi$ over $U/K$ is isomorphic to the 
space of all smooth functions $f: U \to \C$ such that 
\[f  (u  k) = \chi  (k)^{-1}  f (u),\qquad \forall  k \in K,  u \in U.\]
We consider the subspace $C_r^\infty  (U//K;  \mathcal{L}_\chi)$ whose elements satisfy
\[f  (k_1 u k_2) = \chi  (k_1  k_2)^{-1}  f  (u),\qquad \forall  k_1, k_2 \in K,  u \in U\]
and the support of $f$ is contained in a geodesic ball of radius $r$ for some sufficient small $r > 0$. 
The $\chi$-spherical Fourier transform is defined by
\begin{equation}
\label{eq15}
\mathcal{S}  (f)  (\lambda) = \frac{1}{|W|}\int_U  f  (u)  \varphi_{\lambda+\rho,  l}  (u^{-1})  d u,\quad \lambda \in \g{a}_\C^\ast
\end{equation}
where the normalization is chosen to simplify some formulas involving integration over $B^+$ (see below).
In \cite[Theorem 5.1]{hol}, a Paley-Wiener theorem 
for the $\chi$-spherical Fourier transform of these line bundles was proved.
The noncompact case was treated in \cite{sh}.

For the compact group $U$ we have the Cartan decomposition $U = K  B  K$. But $B = \dotcup_{w \in W}  w  \overline{B^+}$ 
with $B^+$ a positive Weyl chamber. The $\chi$-spherical Fourier transform (\ref{eq15}) becomes 
(in view of Proposition \ref{pro2}), up to a constant, 
\begin{equation}
\label{eq16}
\mathcal{S}  (f)  (\lambda) = \int_{B^+}  f  (x)  \eta_l^+  (x)  F  (\lambda + \rho,  m_+ (l);  x)  \delta (x)\, d  x.
\end{equation}
Here, $\delta (x) = \delta (m) (x) = \delta (m, x)$. 
Our goal is to prove that (\ref{eq16}) has at most an exponential growth of type $r$ by using Proposition \ref{pro1}
and some shift operators. Therefore, assume that $f|_B$ (still call it $f$) is compactly supported 
in a geodesic ball of radius  $r$ with $r$ small enough. Define
\[\Delta_s := \prod_{\alpha \in \mathcal{O}_s^+}  |e^\alpha - e^{- \alpha}| = |\delta  (1,  0,  0)|,\]
and similarly, $\Delta_m = |\delta  (0,  1,  0)|$, $\Delta_l = |\delta  (0,  0,  1)|$. 
Note that the absolute value is not needed on $B^+$. In the following we will for simplicity write
$\varphi_\lambda (m;x)$ for $F(\lambda+\rho,m;x)$ and $d := (\frac{\Delta_s}{\Delta_l})^{2  |l|}$.

Recall the notations from Example \ref{exm2}: $k = |l|  b_1 \in \Z^+  \mathcal{B}$, $m' =
 (m_s + 2 |l|,  m_m,  m_l) \in \mathcal{M}^+$, 
and $m_+(l) = m' - k$. By (\ref{eq11}) and Proposition \ref{pro6}, we have up to a constant,
\begin{eqnarray}
 \mathcal{S} (f) (\lambda) & = & \int_{B^+} f(x) \eta_l^+ (x)\varphi_\lambda (m_+(l);x)
\delta  (m, x)\,  d x \label{eq25}\\
& \stackrel{(\ref{eq11})}{=} &  \int_{B^+} f (x) \eta_l^+ (x)  G_- (- k, m') \varphi_\lambda (m';x) \delta (m,x)\, dx \nonumber\\
& = & \int_{B^+} f(x) \eta_l^+ (x) G_- (- k, m') \varphi_\lambda (m';x) (\Delta_s^{m_s}
\Delta_m^{m_m} \Delta_l)(x)\, dx \nonumber\\
& = & \int_{B^+} d(x) f (x) \eta_l^+ (x) G_- (- k, m') \varphi_\lambda (m';x) \nonumber\\
& & \cdot \quad
 (\Delta_s^{m_s - 2 |l|} \Delta_m^{m_m} \Delta_l^{1 + 2 |l|})(x)\, dx \nonumber\\
& = & \int_{B^+}d(x)f(x) \eta_l^+ (x) G_-(- k, m')\varphi_\lambda (m';x) \delta (m_+(l), x)\, dx \label{eq26}\\
& =  & \int_{B^+} G_+ (k, m_+(l)) \left[d(x) f (x) \eta_l^+ (x)\right] \varphi_\lambda (m';x) \delta (m', x)\, d x \, .\nonumber
\end{eqnarray}

Since $m = (m_s, m_m, 1) \in \mathcal{M}^+$, the
term $\delta  (m,  x)$ in (\ref{eq25}) will not blow up and so the integral (\ref{eq25}) is finite. The subsequent four
integrals are simply rewritings of (\ref{eq25}), so they are finite too. Hence, it makes sense to apply Proposition \ref{pro6} to
(\ref{eq26}). Moreover, 
\[\frac{\Delta_s}{\Delta_l} = \prod_{\alpha \in \mathcal{O}_s^+}  \left|\frac{e^\alpha - e^{- \alpha}}{e^{2  \alpha}
- e^{- 2  \alpha}}\right| = \prod_{\alpha \in \mathcal{O}_s^+}  \left|\frac{1}{e^\alpha + e^{- \alpha}}\right|,\]
the denominator is never zero (because $|\alpha  (\,  \cdot\,  )| \leq \pi/4$ for $\alpha \in \mathcal{O}_s^+$), 
so this term will not blow up. It follows that $d$ is a bounded function on $B$. 
Since $m' \in \mathcal{M}^+$, we can use the estimate given in Proposition \ref{pro1}. So there is a $C>0$ such that
\begin{eqnarray*}
|\mathcal{S} (f) (\lambda)| & \leq & \int_{B^+}\left|G_+ (k, m_+(l)) \left[d (x) f (x) \eta_l^+ (x)\right]\right|
|\varphi_\lambda (m';x) | \delta (m', x)\, d x\\
& \leq & C \max_{x \in B^+} \{|d (x)f (x) \eta_l^+ (x)|\} \exp (\max_{w \in W} \re w (\lambda + \rho) (X)) \delta (m', x),
\end{eqnarray*}
where we write $x = e^X$ with $X$ in a ball centered at $0$ of radius $r$ and
$$|\exp(\max_{w \in W} \re w  (\lambda + \rho)  (X))| < e^{r\|\re \lambda\|}.$$
On the other hand, since $f$ has a compact support in a ball of radius $r$, and since $\eta_l^+$ and $d$ are both bounded functions,
applying a differential operator $G_+$ doesn't increase the radius of the ball, giving that 
\begin{equation}
\label{eq17}
\left|G_+ (k,  m_+(l)) \left[d  (x) f (x) \eta_l^+  (x)\right]\right| \leq \max_{x \in B^+}  
\{|d (x) f (x) \eta_l^+(x)|\}.
\end{equation}
The right-hand side of (\ref{eq17}) is actually a constant depending on $r$. 
Moreover, $|\delta  (m',x )|$ is bounded by some constant. Hence, there is a constant $C_1 > 0$ (depending on $r$) 
such that
\[|\mathcal{S}  (f)  (\lambda)| \leq C_1  e^{r  \|\re  \lambda\|}.\]
The polynomial decay of $\mathcal{S} (f)  (\lambda)$ is obtained by applying an invariant differential operator on $G/K$
to $\varphi_{\lambda,  l}$. Hence, $\mathcal{S} (f)  (\lambda)$ has an exponential growth of type $r$.

\section{The Rank One Case}
\label{roc}

\noindent
We treat the rank one case as a simple example for the previous results in Section \ref{appl}.  
In this case the $\chi$-spherical Fourier transform is given by (\ref{eq16}) 
with $d x$ being the invariant measure on the torus $\mathbb{T}$. We give an explicit choice of $G_-$
and show that how easily we can achieve an exponential estimate for $\mathcal{S} (f)  (\lambda)$ in this case.

The rank one case corresponds to $n = \dim  \g{a} = 1$. The only rank one Hermitian compact symmetric space $U/K$ is
\[\mathrm{SU}  (q + 1) / S(\mathrm{U}  (1) \times \mathrm{U} (q))  ,\quad q \geq 1.\]
This is the Grassmann manifold of one-dimensional linear subspaces of
$\C^{q + 1}$. The root system $\Sigma$ is of type $BC_1$. Recall Example \ref{exm1} for more information. 
In this case we have $k_1 = 2  (q - 1)$, $q \geq 1$, and $k_2 = 1$. That is, root multiplicities associated to this symmetric 
space are $(m_s, m_m, m_l) = (2(q-1),  0,  1)$. 
We identify $i  \g{a}$ with $i  \R$, and $\g{a}_\C$ with $\C$. 
So $B = \exp  \g{b} \cong \mathbb{T}$. The Weyl group $W = \{\pm  1\}$ acting on $i  \R$ and $\C$ by multiplication.

A good candidate for $G_-$ is $E_-^{|l|}$ which means repeated applications of $E_-$ 
the number of $|l|$ times (Here $l \in \Z$ and $E_-$ is given by (\ref{eq19})). 
By (\ref{eq11}), 
\begin{eqnarray*}
F  (\lambda+\rho,  m_+(l);  \cdot  ) & = & G_-  (-|l| b_1,  m')  F  (\lambda+\rho,  m';  \cdot  )\\
& = & \underbrace{E_-  (-b_1,  m_+(l)+b_1)  \dotsb  E_-  (-b_1,  m')}_{|l|\; \text{copies}}  F  (\lambda+\rho,  m';  \cdot  )
\end{eqnarray*}

We can write an element $x \in B$ as $x = e^{i  t}$ with $-\pi \leq t \leq \pi$.  
We also need to ensure $f|_B$ compactly supported in a geodesic ball which meets the conditions in Proposition \ref{pro1}, 
so it reduces the domain to $- \pi / 2 \leq t \leq \pi / 2$. 
Note that $f |_B$ is $W$-invariant, so it is a function of $s$ with $s = (x+x^{-1})/2 = \cos  t$. 
We know from Theorem \ref{thm1} that the hypergeometric functions $F  (\lambda,  m;  x)$ are $W$-invariant in $x$. 
A similar argument asserts that $F  (\lambda,  m;  \cdot  )$ can be viewed as a function of $s$.
Since all shift operators are $W$-invariant, in particular, so is $E_-$.
With change of variables $x = e^{i  t}$ and $s = \cos  t$, the operator (\ref{eq19}) is just a first order differential 
operator, having the form
\[E_- = (s - 1)  \frac{d}{d  s} + C\]
where $C$ is the same constant as in (\ref{eq19}).
In the rank one case, $\alpha \in \Sigma^+$ is the only short root, so
\[\eta_l^+  (x) = \eta_l^+  (e^{i  t}) = (\cos  (t / 2))^{2|l|} = \left(\frac{1+s}{2}\right)^{|l|}.\]
Next, the weight measure becomes 
\[|\delta  (m,  x)| \, d x = 2^q  (1 - s)^{q-1}\, d s,\qquad q \geq 1.\]

We choose a positive Weyl chamber $B^+$ such that $0 \leq t \leq \pi / 2$.
Then $0 \leq s \leq 1$, and $d  s = \sin  t\, d t$ is the invariant measure. 
With change of the variable $s = (x+x^{-1})/2$, $0 \leq s \leq 1$, we have, up to a constant, 
\begin{eqnarray}
\mathcal{S} (f) (\lambda) & = & \int_{B^+} f (x) \eta_l^+ (x) [(E_- \dotsb E_-)\varphi_\lambda (m';x)] \delta (m, x) \, d x \nonumber\\
& = & \int_0^1 f (s) (\frac{1 + s}{2})^{|l|} 2^q (1 - s)^{q - 1} [((s - 1) \frac{d}{d s} + C)^{|l|}
\varphi_\lambda (m';x)  ]\, d s. \label{eq18}
\end{eqnarray}
Note that $f (s)$ is smooth compactly supported in a ball of radius $r$, and vanishes on the boundary. Also, the function 
\[2^q  (1 - s)^{q - 1}  (\frac{1 + s}{2})^{|l|}\]
is smooth and bounded. 
Since each $E_-$ is a first order differential operator, we can apply integration by parts ($|l|$ times) to (\ref{eq18})
and then use Proposition \ref{pro1} to get the desired $r$-type exponential growth of $\mathcal{S}  (f)  (\lambda)$.

We mention in the next remark an alternative method to prove $\mathcal{S} (f)  (\lambda)$ has exponential growth of type $r$ 
in some particular higher rank cases. It is a simple and nice generalization of the rank one case.

\begin{remark}
Suppose $U / K$ is a rank $n$ (for $n>1$) Hermitian compact
symmetric space which associates with the multiplicity $(m_s,  m_m,  1)$, where $m_m$ is even. 
Let $\mathfrak{u}$ be the Lie algebra of $U$.  
Let $\mathfrak{u}_j$ be rank one Lie algebras, $j = 1, \dotsc, n$, associated to $(2 (q - 1),  0,  1)$ with $q$
chosen so that $2 (q - 1) = m_s$. Then $\mathfrak{b} = \mathfrak{b}_1 \oplus \dotsb \oplus \mathfrak{b}_n$
is a maximal abelian subspace in $\g{u}$, where each $\mathfrak{b}_j$ is a maximal abelian subspace in $\g{u}_j$. 
Let $B = \exp  \mathfrak{b}$. Then $B = B_1 \times \cdots \times B_n$ with $B_j = \exp  \mathfrak{b}_j$. 
So $x = (x_1, \ldots, x_n) \in B$, $|\delta  (x)| = \prod_{j = 1}^n  |\delta  (x_j)|$, and
\begin{eqnarray*}
F  (\lambda,  m;  x) & = & F  ((\lambda_1, \ldots, \lambda_n),  m;  (x_1, \ldots, x_n))\\
& = & \prod_{j = 1}^n  F  (\lambda_j,  m;  x_j)
\end{eqnarray*}
Let $f \in C_r^\infty  (U//K;  \mathcal{L}_\chi)$. We write $f (x) = f (x_1,  \ldots,  x_n)$, $x \in B$. We have 
(\ref{eq16}) as the $\chi$-spherical Fourier transform of $f$ with $B^+$ a positive Weyl chamber in the higher rank Lie group.

Let $G_+$ be the fundamental shift operator with shift $b_2 = (0, 2, 0)$, repeatedly applied $m_m/2$ times. The existence 
of such a $G_+$ is asserted by Theorem \ref{thm3}. 
The hypergeometric function $F (\lambda, m_+(l);  \, \cdot\, )$ can be 
realized by applying $G_+$ (multiplied by some constant factor) to $F  (\lambda,  (m_s-2|l|, 0, 1+2|l|);  \,\cdot\, )$. Let 
\[G_- = 4^n  (E_-^{|l|})^n = 4^n  E_-^{|l|} \times \dotsb \times E_-^{|l|}\]
where $E_-$ is the same as in (\ref{eq19}) and $E_-^{|l|}$ means applying $E_-$ the number of $|l|$ times. 
The existence of $G_-$ is also followed from Theorem \ref{thm3} (and check \cite[Theorem 3.4.3]{hs}). 
Therefore, 
\begin{eqnarray*}
F  (\lambda,  (m_s-2|l|, 0, 1+2|l|);  x ) & = & G_-  F (\lambda,  m';  x)\\
& = & 4^n  \prod_{j=1}^n  E_-^{|l|}  F  (\lambda_j,  m';  x_j)
\end{eqnarray*}
where $m' = (m_s+2|l|,  0,  1) \in \mathcal{M}^+$. Hence, by using the composition of $G_+$ and $G_-$ we can write $\mathcal{S} (f)  (\lambda)$ as a $n$-fold iterated integral of rank
one integrals with which we have done. The desired $r$-type exponential growth of $\mathcal{S} (f)  (\lambda)$ thus follows from 
the rank one case. 
\end{remark}


\bibliographystyle{amsplain}

\end{document}